\documentclass[a4paper,10pt]{article}
\usepackage{mathtext}
\usepackage[T1,T2A]{fontenc}
\usepackage[cp1251]{inputenc}
\usepackage[english]{babel}
\usepackage{amsmath}
\usepackage{amsfonts}
\usepackage{amssymb}
\usepackage{mathrsfs}
\usepackage{amsthm}
\usepackage{graphicx}
\usepackage{euscript}

\DeclareMathOperator{\sign}{sign}
\DeclareMathOperator{\supp}{supp}

\DeclareMathOperator{\dist}{dist}
\DeclareMathOperator{\Res}{Res}

\newcommand{\eps}{\varepsilon}

\newcommand{\BE}{\mathrm{BE}}

\renewcommand{\leq}{\leqslant}
\renewcommand{\geq}{\geqslant}

\DeclareMathOperator{\vp}{v.p.}
\newcommand{\QE}{\mathrm{QE}}

\newcommand{\scalprod}[2]{\langle{#1},{#2}\rangle}
\newtheorem{Le}{Lemma}[section]
\newtheorem{Def}{Definition}[section]

\newtheorem{Fact}{Fact}
\newtheorem{Th}{Theorem}[section]
\newtheorem{Cor}[Le]{Corollary}
\newtheorem{Rem}[Le]{Remark}
\newtheorem{Conj}{Conjecture}
\numberwithin{equation}{section}

\author{D.~M.~Stolyarov \thanks{Supported by the Chebyshev Laboratory  (Department of Mathematics and Mechanics, St. Petersburg State University) RF Government grant 11.G34.31.0026, by JSC ``Gazprom Neft'', and by RFBR grant no. 14-01-00198 A, and by Rokhlin grant.}}
\title{Bilinear embedding theorems for differential operators in~$\mathbb{R}^2$}
\begin{document}
\maketitle
\begin{abstract}
We prove bilinear inequalities for differential operators in~$\mathbb{R}^2$.  Such type inequalities turned out to be useful for anisotropic embedding theorems for overdetermined systems and the limiting order summation exponent. However, here we study the phenomenon in itself. We consider elliptic case, where our analysis is complete, and non-elliptic, where it is not. The latter case is related to Strichartz estimates in a very easy case of two dimensions. 
\end{abstract}

\tableofcontents

\section{Introduction}\label{SIntro}
The aim of this paper is to provide some generalizations  of the famous Gagliardo--Nirenberg inequality in its easiest form:
\begin{equation}\label{GagliardoNirenberg}
\big|\scalprod{f}{g}_{L_2(\mathbb{R}^2)}\big| \lesssim \|\partial_1f\|_{L_1(\mathbb{R}^2)}\|\partial_2g\|_{L_1(\mathbb{R}^2)}.
\end{equation} 
Here and in what follows we  write ``$a \lesssim b$''  instead of ``$a \leq Cb$ for some uniform constant~$C$'' for brevity; we also write~$a \asymp b$ when~$a \lesssim b$ and~$b \lesssim a$. The symbol~$\partial_j$,~$j=1,2$, denotes the differentiation with respect to the~$j$th variable. 

To be more precise, we study estimates on a scalar product of two functions
in some Hilbert space (in this paper, some Sobolev space of fractional order) in terms of a product of~$L_1$-norms of some differential polynomials applied to these functions.  For the author, the interest in inequalities of
such type originated from the work on non-isomorphism problems for Banach spaces of smooth functions and embedding theorems used there, see the short report~\cite{KMS} and the preprint~\cite{KMSpreprint}.

We are going to use some formalism to make our statements shorter. Suppose~$k$ and~$l$ to be natural numbers,~$\alpha$ and~$\beta$ to be real non-negative numbers, and~$\sigma$ and~$\tau$ to be complex non-zero numbers. By the symbol~$\BE(k,l,\alpha,\beta,\sigma,\tau)$ we mean the statement that the inequality
\begin{equation}\label{BilinearEmbeddings}
\left|\langle f,g\rangle_{W_2^{\alpha,\beta}(\mathbb{R}^2)}\right| \lesssim \left\|(\partial_1^k -\tau\partial_2^l)f\right\|_{L_1(\mathbb{R}^2)}\left\|(\partial_1^k -\sigma\partial_2^l)g\right\|_{L_1(\mathbb{R}^2)}
\end{equation}
holds true for any Schwartz functions~$f$ and~$g$. The scalar product on the left is taken in the fractional Sobolev space,
\begin{equation}\label{norma}
\langle f,g\rangle_{W_2^{\alpha,\beta}(\mathbb{R}^2)}=
\int_{\mathbb{R}^2}\hat{f}(\xi,\eta)\overline{\hat{g}(\xi,\eta)}|\xi|^{2\alpha}|\eta|^{2\beta}\,d \xi d \eta.
\end{equation}
Here and in what follows, by~$\hat \varphi$ we denote the Fourier transform of a function~$\varphi$. When~$\alpha$ and~$\beta$ are integers,~$W^{\alpha,\beta}_2$ reduces to the classical anisotropic homogeneous Sobolev space defined by the semi-norm~$\|\partial_1^{\alpha}\partial_2^{\beta}f\|_{L_2(\mathbb{R}^2)}$. Moreover, if~$p$ and~$q$ are integers, then
\begin{equation}\label{DerivationOfScalarProduct}
\scalprod{\partial_1^p\partial_2^q f}{\partial_1^p\partial_2^q g}_{W^{\alpha,\beta}_2} = (2\pi) ^{2p + 2q} \scalprod{f}{g}_{W_2^{\alpha + p,\beta + q}}.
\end{equation}

Inequalities~\eqref{BilinearEmbeddings} are homogeneous not in a usual sense, they are \emph{anisotropic} homogeneous. The transformations that preserve the right-hand side of~\eqref{BilinearEmbeddings}
 (and thus have to preserve the left-hand side) are not of the form~$(x,y) \to (\lambda x, \lambda y)$, but of the form~$(x,y)\to(\lambda^l x,\lambda^k y)$. The basic theory of such type embedding theorems can be found, e.g. in~\cite{BIN} (see also~\cite{Kol} and~\cite{S}, where the case of the limiting order summation exponent is considered). Using this rescaling, we see that~$\BE(k,l,\alpha,\beta,\sigma,\tau)$ may hold true only when
\begin{equation}\label{line}
\frac{\alpha + \frac{1}{2}}{k} + \frac{\beta + \frac{1}{2}}{l} = 1.
\end{equation}

It may seem that the validity of~$\BE$ should not depend on~$\sigma$ and~$\tau$ too much. However, at least one point is obvious from the very beginning:  \emph{elliptic} cases should be distinguished from non-elliptic ones.
\begin{Def}\label{Ellipticity}
We say that the set of numbers~$(k,l,\sigma,\tau)$ is elliptic if both polynomials~$\xi^k - \sigma_1 \eta^l$ and~$\xi^k - \tau_1 \eta^l$\textup, $\tau_1=(2\pi i)^{l-k}\tau$\textup,~$\sigma_1=(-1)^{l-k}(2\pi i)^{l-k}\overline{\sigma}$\textup, do not have real roots except~$0$. 
\end{Def}
Here and in what follows,~$\tau_1$ always denotes~$(2\pi i)^{l-k}\tau$,~$\sigma_1$ always denotes~$(-1)^{l-k}(2\pi i)^{l-k}\overline{\sigma}$.
In the case where one of the numbers~$k$ and~$l$ is odd, ellipticity means that the numbers~$\tau_1$ and~$\sigma_1$ are not real. When both~$k$ and~$l$ are even, the numbers~$\sigma_1$ and~$\tau_1$ may be real negative. We begin with the elliptic case. 
\begin{Th}\label{T3}
Suppose that~$k$ and~$l$ are natural and~$(k,l,\sigma,\tau)$ is elliptic. The statement~$\BE(k,l,\alpha,\beta,\sigma,\tau)$ holds true if and only if one of the numbers~$k$ and~$l$ is odd\textup,~$\alpha = \frac{k-1}{2}$\textup,~$\beta=\frac{l-1}{2}$\textup, and the numbers~$\sigma_1$ and~$\tau_1$ have non-zero imaginary parts of the same sign.
\end{Th}
The ``if'' part of this theorem was proved in the preprint~\cite{KMS}. However, we repeat it here both for the sake of completeness and because it serves as a good introduction to a similar theorem for the non-elliptic case. Theorem~\ref{T3} has a special particular case~$\sigma = \tau$.
\begin{Cor}\label{sigmaEqualstau}
The statement~$\BE(k,l,\alpha,\beta,\sigma,\sigma)$ does not hold if~$(k,l,\sigma,\sigma)$ is elliptic.
\end{Cor}
Moreover, it will follow from the proof that the inequality
\begin{equation*}
\|f\|_{W^{\alpha,\beta}_2} \lesssim \big\|(\partial_1^k - \sigma\partial_2^l)\big\|_{L_1}
\end{equation*}
does not hold when~$(k,l,\sigma,\sigma)$ is elliptic. 
Theorem~\ref{T3} is, in some sense, a surprise. Usually, if one has an embedding into~$W^{\frac{k-1}{2},\frac{l-1}{2}}_2$, then it is natural to expect that it is also valid for all other pairs~$(\alpha,\beta)$ that satisfy equation~\eqref{line} (for example, all the embedding theorems from~\cite{Kol} and~\cite{S} are of this type). However, here this principle does not work, mainly due to the fact that we are working with bilinear embeddings.

The limiting case~$\sigma = \infty,\tau = 0$ was considered in~\cite{PS}, where the inequality
\begin{equation*}
\left|\langle f,g\rangle_{W_2^{\frac{k-1}{2},\frac{l-1}{2}}(\mathbb{R}^2)}\right| \lesssim \left\|\partial_1^k f\right\|_{L_1(\mathbb{R}^2)}\left\|\partial_2^l g\right\|_{L_1(\mathbb{R}^2)}
\end{equation*}
was proved. As a consequence, this inequality leads to the inequality 
\begin{equation*}
\|f\|_{W^{\frac{k-1}{2},\frac{l-1}{2}}_2}^2 \lesssim \|\partial_1^k f\|_{L_1}\|\partial_2^l f\|_{L_1},
\end{equation*}
which is a very particular case of Solonnikov's embedding theorem, see~\cite{S} (see also~\cite{Kol} for even stronger results and~\cite{KS} for a detailed explanation why this inequality follows from much more general theorems in~\cite{S}). By rescaling, the latter inequality is equivalent to
\begin{equation*}
\|f\|_{W^{\frac{k-1}{2},\frac{l-1}{2}}_2} \lesssim \|\partial_1^k f\|_{L_1} + \|\partial_2^l f\|_{L_1} \asymp \left\|(\partial_1^k -\tau\partial_2^l)f\right\|_{L_1} + \left\|(\partial_1^k -\sigma\partial_2^l)g\right\|_{L_1}
\end{equation*}
provided~$\sigma \ne \tau$. We signalize that~$\BE(k,l,\frac{k-1}{2},\frac{l-1}{2},\sigma,\tau)$ does not follow from this inequality (because by Theorem~\ref{T3} the former statement may not hold);  the multiplicative and the additive forms are not equivalent here.

When we consider non-elliptic cases and~$k \ne l$ (the case~$k=l$ is special), we are working with something like the Strichartz estimates (see, e.g.~\cite{T}). This hints us that the treatment of non-elliptic cases will lead us to some oscillatory integrals (fortunately, one-dimensional) that are unavoidable when dealing with the Strichartz estimates. However, the main difficulty is still the~$L_1$-space on the right-hand side of~\eqref{BilinearEmbeddings}. 
\begin{Th}\label{TVanishing}
Suppose that~$k$ and~$l$ are natural\textup, one of them is odd\textup, and~$k\ne l$. Suppose that~$\sigma$ and~$\tau$ are complex numbers such that both numbers~$\tau_1=(2\pi i)^{l-k}\tau$ and~$\sigma_1=(-1)^{l-k}(2\pi i)^{l-k}\overline{\sigma}$ are real. In this case\textup,~$\BE(k,l,\frac{k-1}{2},\frac{l-1}{2},\sigma,\tau)$ holds true.
\end{Th}
A simple algebraic trick leads to some positive results for the case where both~$k$ and~$l$ are even.
\begin{Cor}\label{EvenParametersVanishingOp}
Suppose that both numbers~$k$ and~$l$ are even\textup, but one of them contains~$2$ only in the first power. In this case\textup, the statements~$\BE(k,l,\frac{3}{4} k -\frac12,\frac14 l - \frac12,\sigma,\tau)$ and~$\BE(k,l,\frac14 k -\frac12,\frac34 l - \frac12,\sigma,\tau)$ hold true whenever~$\sigma_1$ and~$\tau_1$ are distinct positive real numbers.
\end{Cor}

We also have some negative results for non-elliptic operators. 
\begin{Le}\label{KnappExample}
The statement~$\BE(k,l,\alpha,\beta,\sigma,\tau)$ does not hold when~$\sigma = \tau$ and~$k\ne l$.
\end{Le}
This lemma will be an easy consequence of a rather standard trick called the Knapp example~(see~\cite{T}). We did not find an exact reference for our case, so we will repeat this standard stuff. We signalize that the effect that breaks~$\BE$ in this case is totally different from the one we had in Theorem~\ref{T3}. In Lemma~\ref{KnappExample} it is of oscillatory integral nature, whereas in the said theorem it is a singular integral that ``diverges''. We specialize the case~$k\ne l$, because it is more difficult and interesting. Surely, the conclusions of~Theorem~\ref{TVanishing}, Corollary~\ref{EvenParametersVanishingOp}, and~Lemma~\ref{KnappExample} hold true (with a bit different but easier proofs) when~$k=l$. We leave this cases.

The text is organized as follows: in Section~\ref{SElliptic} we deal with the elliptic case (prove Theorem~\ref{T3}); in Section~\ref{SVanish} we work on non-ellipticity cases; finally, in Section~\ref{SFurther}, we give an explanation why our results for non-elliptic cases are not complete and give some related results and conjectures for linear and quadratic inequalities. In this last section we omit detailed proofs, but provide sketches (the results presented there are incomplete, they need further study; moreover, this is a related, but different story).

This paper is of technical character: we believe that neither the statements nor the proofs presented here are general enough (the worst thing is that we are working not with many variables, but with only two of them). However, the effect itself (especially for non-elliptic cases) seems to be new in its nature. Both these facts urge us to write down too many details in Sections~\ref{SElliptic} and~\ref{SVanish}, we apologize for that.

I am grateful to my scientific adviser S. V.~Kislyakov for statement of the problem and attention to my work and to A. I. Nazarov for useful comments.

\section{Elliptic case}\label{SElliptic}
We are going to prove Theorem~\ref{T3} in the following order: first, we prove that~$\BE$ holds if the parameters satisfy the assumptions; second, we construct counterexamples that disprove~$\BE$ in all the remaining cases. Unfortunately, we did not manage to formalize the latter part in a good way: though the ideology of the counterexample is clear, the technical treatment varies for different cases of parameters.

\subsection{Proof of~$\BE$ in the case of an odd number~$k$}\label{SsOddk} 
At least one of the numbers~$k$ and~$l$ is odd; by symmetry, we may suppose that~$k$ is odd.  In what follows, we assume that~$\sigma_1$ and~$\tau_1$ are distinct (the case where they are equal is a bit different, we treat it afterwards).
The functions~$f$ and~$g$ belong to the Schwartz class, so, their Fourier transform is infinitely differentiable and decays rapidly at infinity. Thus, we can represent the scalar product on left-hand side of~\eqref{BilinearEmbeddings}
as a limit,
\begin{equation*}
\int\limits_{\mathbb{R}^2}\hat{f}(\xi,\eta)\overline{\hat{g}(\xi,\eta)}\xi^{k-1}|\eta|^{l-1}\,d \xi d \eta =
\lim_{\substack{\varepsilon\rightarrow 0 \\ R \rightarrow \infty}}\,\int\limits_{\Omega_{\varepsilon,R}} \hat{f}(\xi,\eta)\overline{\hat{g}(\xi,\eta)}\xi^{k-1}|\eta|^{l-1}\,d \xi d \eta,
\end{equation*}
where~$\Omega_{\varepsilon,R}=\{(\xi,\eta)\in\mathbb{R}^2\mid \varepsilon\leqslant |\eta|\leqslant R\}$. 
Next, we replace the Fourier transforms in the integrand by their expressions
in terms of~$f_1$ and~$g_1$ found from the formulas
\begin{equation*}
\big((2\pi i\xi)^k - \tau (2\pi i\eta)^l\big)\hat{f}(\xi,\eta) =
\hat{f}_1(\xi,\eta); \quad \big((2\pi i\xi)^k - \sigma (2\pi i\eta)^l\big)\hat{g}(\xi,\eta) = \hat{g}_1(\xi,\eta),
\end{equation*}
so, the expression~$\|f_1\|_{L_1}\|g_1\|_{L_1}$ coincides with the right-hand side of~\eqref{BilinearEmbeddings}. We must estimate the quantity
\begin{equation*}
\lim_{\substack{\varepsilon\rightarrow 0 \\ R \rightarrow \infty}}\,
\int\limits_{\Omega_{\varepsilon,R}}\frac{|\eta|^{l - 1}\xi^{k - 1}\hat{f}_1(\xi,\eta)\overline{\hat{g}_1(\xi,\eta)} \,d\xi d\eta}{\big((2\pi i\xi)^k -
\tau (2\pi i\eta)^l\big)\big((-2\pi i\xi)^k - \overline{\sigma} (-2\pi i\eta)^l\big)}.
\end{equation*}
The denominator of the integrand does not vanish on~$\mathbb{R}^2$ except at zero, due to the assumptions of ellipticity.
Then (for~$\varepsilon$ and~$R$ fixed) we replace~$\hat{f}_1$ and~$\hat{g}_1$ in the last
formula by their definitions in terms of~$f_1$ and~$g_1$, and change the order of integration:
\begin{equation}\label{quantity}
\lim_{\substack{\varepsilon \rightarrow 0 \\ R \rightarrow \infty}}
\iint F(\varepsilon,R,x_1,x_2,y_1,y_2)
f_1(x_1,x_2)\overline{g_1(y_1,y_2)} \,dy_1 dy_2 dx_1 dx_2,
\end{equation}
where
\begin{equation}\label{integrand}
F(\varepsilon,R,x_1,x_2,y_1,y_2)=\int\limits_{\Omega_{\varepsilon,R}}\frac{|\eta|^{l-1}\xi^{k-1}
e^{2\pi i ((x_1 - y_1)\xi + (x_2 - y_2)\eta)}}{((2\pi i\xi)^k - \tau (2\pi i\eta)^l)((-2\pi i\xi)^k -
\overline{\sigma} (-2\pi i\eta)^l)} d\xi d\eta.
\end{equation}

To prove inequality~\eqref{BilinearEmbeddings}, we must show that the modulus of the quantity
\eqref{quantity} does not exceed the value~$C\|f_1\|_1\|g_1\|_1$. For this, we prove that the function
\eqref{integrand} is bounded uniformly in all~$\eps$ and~$R$, $0<\varepsilon\leqslant R<\infty$, and almost all quads of reals~$x_1,\,x_2,\,y_1$, $y_2$. The things become slightly more transparent if we integrate in~\eqref{integrand} first with respect to~$\xi$ and then with respect to~$\eta$, and in the inner integral introduce
a new variable~$\rho$ by setting~$\xi=\rho |\eta|^{l/k}$. This yields the formula
\begin{equation*}
F(\varepsilon,R,x_1,x_2,y_1,y_2)=\frac{1}{(2\pi)^{2k}}\!\!\!\!\int\limits_{\varepsilon\leqslant |\eta|\leqslant R}
\!\!\!\!\!|\eta|^{-1}\!\!\!\int\limits_{-\infty}^{+\infty}\frac{\rho^{k-1}e^{2\pi i(a\rho|\eta|^{l/k}+b\eta)}}{\big(\rho^k -
\tau_1(\sign\eta)^l\big)\big(\rho^k -\sigma_1(\sign\eta)^l\big)}\,d\rho d\eta,
\end{equation*}
where we have put~$a=x_1 - y_1$,~$b=x_2-y_2$, and~$\tau_1$ and~$\sigma_1$ were introduced in Definition~\ref{Ellipticity}.

Here integration in~$\eta$ is over the union~$[-R,-\varepsilon]\cup [\varepsilon,R]$, and we will prove
the boundedness of the integral over each of these two intervals separately. Because of symmetry, we
consider only the interval~$[\varepsilon,R]$ (then~$\sign\eta$ in the denominator disappears). For definiteness,
we assume that $a>0$ (the opposite case reduces to this one if we change $\rho$ by $-\rho$; note that we may drop the case of $a=0$ because
it corresponds to a set of measure $0$). After that, we take
$(2\pi a)^{k/l}\eta$ for a new variable in the outer integral; this will modify the parameters $b$, $\varepsilon$, and $R$,
but will allow us to assume that $2\pi a=1$. So, finally, we must show that the following integral is bounded uniformly
in~$\eps$,~$R$,~$0\leqslant\varepsilon <R$, and~$b$:
\begin{equation}\label{ff}
\int\limits_{\varepsilon}^R \eta^{-1}\left(\int\limits_{-\infty}^{\infty}
\frac{\rho^{k-1}e^{i(\rho|\eta|^{l/k}+b\eta)}}{(\rho^k -\tau_1)(\rho^k -\sigma_1)}\,d\rho\right)\,d\eta.
\end{equation}
We will calculate the
integral with respect to~$\rho$ with the
help of the residue formula, perceiving~$\rho$ as a complex variable. Since the integrand
decays rapidly at infinity in the upper half-plane, integration over the contour that consists of the
interval~$[-r, r]$ and the upper part of the circle of radius~$r$ and centered at zero shows,
after the limit passage as~$r\to\infty$, that the integral in question is~$2\pi i$ times the
sum of the residues at the poles of the integrand in the upper half-plane. All these poles are simple and are~$k$th
roots of~$\tau_1$ or~$\sigma_1$ (we have assumed that~$\sigma_1 \ne \tau_1$). Let~$u^k=\tau_1$ (and~$\Re u > 0$). Perceiving the integrand in question as~$\frac{\varphi(\rho)}{\psi(\rho)}$ with~$\psi(\rho)=\rho^k -\tau_1$, we use the fact that~$\varphi$ is regular at~$u$ to conclude that the residue at~$u$ is
\begin{equation*}
\frac{\varphi(u)}{\psi^{\prime}(u)}=\frac{u^{k-1}e^{i(u|\eta|^{l/k}+b\eta)}}{ku^{k-1}(\tau_1 -\sigma_1)}
=\frac{e^{i(u|\eta|^{l/k}+b\eta)}}{k(\tau_1 -\sigma_1)}.
\end{equation*}
Similarly, if~$v^k=\sigma_1$ (and~$\Re v>0$), the residue at~$v$ is equal to
\begin{equation*}
\frac{e^{i(v|\eta|^{l/k}+b\eta)}}{k(\sigma_1 -\tau_1)}.
\end{equation*}
Under our assumptions on~$\sigma$ and~$\tau$, the equations~$\rho^k=\tau_1$ and~$\rho^k=\sigma_1$ have one and the same number of
roots in the upper half plane. This shows that, up to a constant factor, the integral~\eqref{ff} is equal
to a sum of~$(k\pm 1)/2$ expressions of the form
\begin{equation}\label{SimpleIntegral}
\frac{1}{k(\tau_1-\sigma_1)}\int\limits_{\varepsilon}^R e^{ib\eta}\frac{e^{iu|\eta|^{l/k}}-e^{iv|\eta|^{l/k}}}{\eta}d\eta,
\end{equation}
where~$u$ and~$v$ are some~$k$th roots of, respectively,~$\tau_1$ and~$\sigma_1$ in the upper half-plane
(note that, happily, the denominators in the above two formulas for residues are opposite to each other).

We recall that~$b$ is real, so we estimate the absolute value of the integrand by~$C|u-v|\eta^{\frac lk -1}$ when~$\eta \leq 1$ and by~$e^{-\Im u |\eta|^{l/k}}+e^{-\Im v |\eta|^{l/k}}$ when~$\eta > 1$. Both integrals~$\int_0^1 \eta^{\frac lk - 1}\,d\eta$ and~$\int_1^{\infty}(e^{-\Im u |\eta|^{l/k}}+e^{-\Im v |\eta|^{l/k}}) \,d\eta$ converge. This finishes the proof
of~$\BE(k,l,\frac{k-1}{2},\frac{l-1}{2},\sigma,\tau)$ when~$k$ is odd and~$\sigma_1$ and~$\tau_1$ are distinct numbers with imaginary parts of the same sign.

The case where~$\sigma_1 = \tau_1$ can be treated in a similar way. In this case,  all the poles of the function~$\rho \mapsto\frac{\rho^{k-1}e^{i(\rho|\eta|^{l/k}+b\eta)}}{(\rho^k -\tau_1)(\rho^k -\sigma_1)}$ are of the second order. After calculations, the two-dimensional integral~\eqref{ff} appears to be a linear combination of one-dimensional integrals of the form
\begin{equation*}
\int\limits_{\varepsilon}^R e^{ib\eta}|\eta|^{\frac{l}{k} - 1}e^{iu|\eta|^{l/k}}d\eta.
\end{equation*}
Alternatively, this can be justified by passing to the limit as~$\tau_1 \to \sigma_1$ in formula~\eqref{SimpleIntegral}. This integral is uniformly bounded, again because~$\Im u > 0$.
\begin{Rem}
In fact\textup, we have proved that the distribution~$M_{k,l,\sigma,\tau}$ defined by formula
\begin{equation}\label{Distribution}
\scalprod{M_{k,l,\sigma,\tau}}{\phi} = \lim\limits_{\substack{R \to \infty\\ \eps \to 0}}\,
\int\limits_{\Omega_{\varepsilon,R}}\frac{|\eta|^{l - 1}\xi^{k - 1}\phi(\xi,\eta) \,d\xi d\eta}{\big((2\pi i\xi)^k -
\tau (2\pi i\eta)^l\big)\big((-2\pi i\xi)^k - \overline{\sigma} (-2\pi i\eta)^l\big)},\quad \phi \in \mathcal{S}(\mathbb{R}^2),
\end{equation}
has bounded Fourier transform. We did not formulate this as a statement\textup, because it is not clear a priori why does formula~\textup{\eqref{Distribution}} define a distribution. By the classical Malgrange--Ehrenpreis theorem
\begin{equation*}
\frac{1}{\big((2\pi i\xi)^k -
\tau (2\pi i\eta)^l\big)\big((-2\pi i\xi)^k - \overline{\sigma} (-2\pi i\eta)^l\big)}
\end{equation*}
 is a distribution belonging to~$\mathfrak{D}'(\mathbb{R}^2)$ \textup(and even to~$\mathcal{S}'(\mathbb{R}^2)$\textup, see~\textup{\cite{H}}\textup)\textup, however\textup, it is not clear whether it can be multiplied by a non-smooth function~$|\eta|^{l-1}$.
\end{Rem}

\subsection{Counterexamples}\label{SsCounterexamples}
Assume that~$\BE(k,l,\alpha,\beta,\sigma,\tau)$ holds. In this case, we know that~$\alpha$ and~$\beta$ must satisfy equation~\eqref{line}. In this subsection, we construct a pair of functions~$f$ and~$g$ (depending on the parameters) that will give further restrictions on the numbers~$k,l,\alpha,\beta,\sigma,\tau$. We also suppose that if at least one of the numbers~$k$ and~$l$ is even, then~$l$ is even.

Let~$\varphi$,~$0 \leq \varphi \leq 1$, be an infinitely differentiable compactly supported function on~$\mathbb{R}$ equal to~$1$ near zero.
We put~$\psi(\xi,\eta) = \varphi\big(\sqrt{\xi^{2k}+\eta^{2l}}\big)$ and~$\psi_t(\xi,\eta) = \psi(\frac{\xi}{t^l},\frac{\eta}{t^k})$, where~$t > 0$
should be thought of as a large number. Next, let~$V = \psi_t - \psi$, and let~$h = \check{V} \in \mathcal{S}(\mathbb{R}^2)$. 
Then the~$L_1(\mathbb{R}^2)$-norm of~$h$ is dominated by a constant independent of~$t$. At the same time,~$\hat{h} = V$ is equal 
to~$0$ near the origin and is equal to~$1$ on a large set if~$t$ is large. Note also that~$V(\xi,\eta) = v(\sqrt{\xi^{2k} + \eta^{2l}})$ for
some~$v \in \mathfrak{D}(\mathbb{R})$.

Now, we find two functions~$F,G \in \mathcal{S}(\mathbb{R}^2)$ from the equations
\begin{equation*}
(\partial_1^k - \tau\partial_2^l)F = (\partial_1^k - \sigma\partial_2^l)G = h.
\end{equation*}
These equations are easily solvable after passage to Fourier transforms:
\begin{equation*}
\hat{F}(\xi,\eta) = \frac{V(\xi,\eta)}{(2\pi i \xi)^k - \tau(2\pi i \eta)^l},\quad \hat{G}(\xi,\eta) = \frac{V(\xi,\eta)}{(2\pi i \xi)^k - \sigma(2\pi i \eta)^l},
\end{equation*}
and it is clear that the solutions are in the Schwartz class indeed (recall that the polynomials in the denominators do not have zeros except~$0$).

Now, inequality~\eqref{BilinearEmbeddings} with these~$F$ and~$G$ yields
\begin{equation*}
\bigg|\int\limits_{-\infty}^{+\infty}\int\limits_{-\infty}^{+\infty}\frac{|V(\xi,\eta)|^2|\xi|^{2\alpha}|\eta|^{2\beta}\,d\xi d\eta}{\big((2\pi i \xi)^k - \tau(2\pi i \eta)^l\big)\big((2\pi i \xi)^k - \sigma(2\pi i \eta)^l\big)}\bigg| \lesssim 1
\end{equation*}
independently of~$t$. We transform the integral as we did it before, 
including the change of variables~$\xi = \rho|\eta|^{\frac{l}{k}}$, to obtain
\begin{equation*}
\bigg|\int\limits_{0}^{+\infty}|\eta|^{-1}\int\limits_{-\infty}^{+\infty} \frac{V\big(\rho|\eta|^{\frac{l}{k}},\eta\big)^2|\rho|^{2\alpha}}{\big(\rho^k - \tau_1\big)\big(\rho^k - \sigma_1\big)} \, d\rho d\eta\bigg| \lesssim 1.
\end{equation*}
The power~$|\eta|^{-1}$ in the outer integral arises after a short calculation involving equation~\eqref{line}, and the signs of~$\eta$ in the denominator disappear, if~$l$ is even (if both~$k$ and~$l$ are odd, then one has twice the same integral after splitting the integral with respect to~$\eta$ into integrals over~$(-\infty,0]$ and~$[0,\infty)$ and making the change of variable~$\rho \to -\rho$ in the first one). Changing the order of integration, we arrive at 
\begin{equation*}
\int\limits_{-\infty}^{+\infty}\bigg[\frac{|\rho|^{2\alpha}}{\big(\rho^k - \tau_1\big)\big(\rho^k - \sigma_1\big)}\int\limits_{0}^{+\infty}\frac{v^2\big(\sqrt{\rho^{2k}\eta^{2l} +\eta^{2l}}\big)}{\eta}\,d\eta\bigg]d\rho.
\end{equation*} 

An obvious change of variable shows that the integral with respect to~$\eta$ is equal to~$\int_{0}^{\infty}\frac{v^2(u^l)}{u}\,du$, which does not depend
on~$\rho$ and can be made arbitrarily large if~$t$ is large. So,~$\BE(k,l,\alpha,\beta,\sigma,\tau)$ can only be fulfilled if
\begin{equation}\label{a3}
\int\limits_{-\infty}^{\infty}\frac{|\rho|^{2\alpha}\,d\rho}{\big(\rho^k - \tau_1\big)\big(\rho^k - \sigma_1\big)} = 0.
\end{equation}

\subsubsection{Cases where~$k$ is even (both numbers~$k$ and~$l$ are even)}
In this case the integral~\eqref{a3} can be rewritten as
\begin{equation*}
2\int\limits_{0}^{\infty}\frac{|\rho|^{2\alpha}\,d\rho}{\big(\rho^k - \tau_1\big)\big(\rho^k - \sigma_1\big)} = \frac{2}{k}\int\limits_{0}^{\infty}\frac{|\rho|^{\frac{2\alpha-k+1}{k}}\,d\rho}{\big(\rho - \tau_1\big)\big(\rho - \sigma_1\big)}.
\end{equation*}
Both numbers~$k$ and~$l$ are even, so, the situation is symmetric and we may assume that~$\alpha \leq \frac{k-1}{2}$. The case~$\alpha = \frac{k-1}{2}$ differs a bit from the other cases.

\paragraph{Case~$\alpha = \frac{k-1}{2}$ and~$\sigma_1 \ne \tau_1$.} In this case, the integral in question can be computed directly:
\begin{equation*}
\int\limits_{0}^{\infty}\frac{\,d\rho}{\big(\rho - \tau_1\big)\big(\rho - \sigma_1\big)}= \frac{\log (-\sigma_1) - \log(-\tau_1)}{\sigma_1 - \tau_1} \ne 0,
\end{equation*}
provided~$\sigma_1 \ne \tau_1$. Here~$\log$ is the branch of logarithm defined for~$\arg z \in [0,2\pi)$ such that it is real when~$\arg z = 0$. 
\paragraph{Case~$\alpha = \frac{k-1}{2}$ and~$\sigma_1 = \tau_1$.} If~$\sigma_1 = \tau_1$, then the integral~\eqref{a3} equals~$\tau_1^{-1}$, which is non-zero too.
\paragraph{Case~$\alpha < \frac{k-1}{2}$,~$\sigma_1 \ne \tau_1$.} In this case, the integral can be rewritten as
\begin{equation*}
\int\limits_{0}^{\infty}\frac{|\rho|^{\frac{2\alpha-k+1}{k}}\,d\rho}{\big(\rho - \tau_1\big)\big(\rho - \sigma_1\big)} = \frac{1}{\sigma_1 - \tau_1}\bigg(\int\limits_{0}^{\infty}\frac{|\rho|^{\frac{2\alpha-k+1}{k}}\,d\rho}{\rho - \sigma_1} - \int\limits_{0}^{\infty}\frac{|\rho|^{\frac{2\alpha-k+1}{k}}\,d\rho}{\rho - \tau_1}\bigg).
\end{equation*}
These integrals converge absolutely, because we have assumed that~$\frac{2\alpha -k + 1}{k} < 0$.
Thus, we have to prove that the function~$\Phi$ given by the formula
\begin{equation*}
\Phi(\zeta) = \int\limits_{0}^{\infty}\frac{|\rho|^{\frac{2\alpha-k+1}{k}}\,d\rho}{\rho - \zeta}
\end{equation*}
is an injective function on~$\mathbb{C}\setminus\mathbb{R}_+$. We list the properties of the function~$\Phi$.
\begin{enumerate}
\item The function~$\Phi$ is holomorphic as a function on~$\mathbb{C}\setminus\mathbb{R}_+$.
\item The function~$\Phi$ is non-zero, e.g. it is positive on~$\mathbb{R}_-$.
\item For any~$\lambda \in \mathbb{R}_+$ one has~$\Phi(\lambda\zeta) = \lambda^{\frac{2\alpha - k+ 1}{k}}\Phi(\zeta)$.
\end{enumerate} 
The third property follows by a change of variable in the integral that defines~$\Phi$. Let~$\Delta(z)$ be the branch of~$z \mapsto z^{\frac{2\alpha -k+1}{k}}$ defined on~$\mathbb{C}\setminus\mathbb{R}_+$ that is real positive on the negative real  half-line. By the third property of the function~$\Phi$,~$\Phi(\zeta) = \Phi(-1)\Delta(\zeta)$ when~$\zeta$ is a negative real. By the uniqueness theorem for analytic functions,~$\Phi(\zeta) = \Phi(-1)\Delta(\zeta)$ for~$\zeta \in \mathbb{C} \setminus\mathbb{R}_+$. But the function~$\Delta$ is injective, because~$\frac{2\alpha -k + 1}{k} > -1$ due to equation~\eqref{line}. 
\paragraph{Case~$\alpha < \frac{k-1}{2}$,~$\sigma_1 = \tau_1$.} To treat the remaining case, we note that the integral~\eqref{a3} in question in this case equals~$\Phi'(\sigma_1)$ (where~$\Phi$ is the holomorphic function introduced above). This value is non-zero, because a holomorphic injective function cannot have vanishing derivative. 
\subsubsection{Cases where~$k$ is odd}
In these cases we have~$|\rho|^{k-1} = \rho^{k-1}$,
and~\eqref{a3} can be rewritten as follows ($s = \rho^k$):
\begin{equation*}
\int\limits_{-\infty}^{\infty} \frac{|\rho|^{2\alpha}}{\big(\rho^k - \tau_1\big)\big(\rho^k - \sigma_1\big)}\,d\rho=
\frac{1}{k}\int\limits_{-\infty}^{\infty} \frac{|s|^{\frac{2\alpha - k + 1}{k}}}{(s-\tau_1)(s - \sigma_1)}\,ds.
\end{equation*}
\paragraph{Case~$\alpha = \frac{k-1}{2}$,~$\sigma_1$ and $\tau_1$ have imaginary parts of different signs.}
In this case the modulus disappears, we are integrating an analytic function. Integrating over the same contour as we did in Subsection~\ref{SsOddk}, we see that the integral equals zero if~$\sigma_1$ and~$\tau_1$ have imaginary parts of the same sign, and does not equal zero if they do not.
\paragraph{Case~$\alpha \ne \frac{k-1}{2}$,~$\sigma_1 \ne \tau_1$.}
We do the same trick as before and rewrite the integral as
\begin{equation*}
\frac{1}{\tau_1 - \sigma_1}\lim\limits_{R \to +\infty}\int\limits_{-R}^R\Big(\frac{|s|^{\frac{2\alpha-k+1}{k}}}{s-\tau_1} - \frac{|s|^{\frac{2\alpha-k+1}{k}}}{s-\sigma_1}\Big)\, ds
\end{equation*}
and introduce the function~$\Phi$ given by formula
\begin{equation*}
\Phi(\zeta) = \lim\limits_{R \to +\infty}\int\limits_{-R}^R\frac{|s|^{\frac{2\alpha-k+1}{k}}}{s - \zeta}\,ds.
\end{equation*}
Here are the properties of the functions~$\Phi$ (note that we still have to prove that the limit in~$R$ exists).
\begin{enumerate}
\item The function~$\Phi$ is analytic in~$\mathbb{C} \setminus \mathbb{R}$ (in particular, the limit in~$R$ exists).
\item The function~$\Phi$ is non-zero, e.g.  it is non-zero on the imaginary axis (it is pure imaginary there).
\item For all admissible~$\zeta$,~$\Phi(-\zeta) = -\Phi(\zeta)$.
\item For any~$\lambda \in \mathbb{R}_+$ one has~$\Phi(\lambda\zeta) = \lambda^{\frac{2\alpha - k+ 1}{\alpha}}\Phi(\zeta)$.
\end{enumerate} 
Only the first property needs a proof (together with the fact that the function~$\Phi$ is well-defined). Writing
\begin{equation*}
\lim\limits_{R \to +\infty}\int\limits_{-R}^R\frac{|s|^{\frac{2\alpha-k+1}{k}}}{s - \zeta}\,ds = \int\limits_{0}^R s^{\frac{2\alpha - k + 1}{k}}\Big(\frac{1}{s - \zeta} - \frac{1}{s+ \zeta}\Big)\,ds = \int\limits_{0}^R \frac{2\zeta s^{\frac{2\alpha - k + 1}{k}}}{(s - \zeta)(s+ \zeta)}\,ds,
\end{equation*}
we again obtain an absolutely convergent integral, because~$\frac{2\alpha - k + 1}{k} < 1$ by virtue of equality~\eqref{line}.

We must prove that~$\Phi$ is injective on~$\mathbb{C} \setminus\mathbb{R}$. We denote by~$\Delta$
the branch of the power~$z^{\frac{2\alpha - k + 1}{k}}$ that arises if we allow the argument to vary within~$(-\pi,\pi]$ and maps~$\mathbb{R}_+
$ to itself.
By the fact that~$|\frac{2\alpha - k + 1}{k}| < 1$,~$\Delta$ is injective. Moreover, since~$\Delta$ takes the upper half-plane either into 
the upper or into the lower half-plane (depending on the sign of~$2\alpha -k +1$), we see that
\begin{equation}\label{a6}
\Delta(\zeta_1) + \Delta(\zeta_2) \ne 0
\end{equation}
whenever~$\zeta_1,\zeta_2 \in \mathbb{C}_{+}$.

Now, consider a point~$\zeta_0$ that lies inside the first quarter. If~$|\arg \lambda| < \delta$, where~$\delta$ is
sufficiently small, then~$\lambda\zeta_0 \in \mathbb{C}_+$, and the function~$\lambda \mapsto \Phi(\lambda \zeta_0)$ 
is analytic. But for~$\lambda > 0$ we have
\begin{equation*}
\Phi(\lambda\zeta_0) = \Delta(\lambda)\Phi(\zeta_0)
\end{equation*}
by the fourth property of~$\Phi$,
whence~$\Phi(\lambda \zeta_0) = C\Delta(\lambda \zeta_0)$ for~$|\arg \lambda| < \delta$. 
By the uniqueness theorem,~$\Phi(\zeta) = C\Delta(\zeta)$ for all~$\zeta \in \mathbb{C}_{+}$. Clearly,~$C \ne 0$ because~$\Phi$ is nonzero on
the imaginary axis.

Now, we see that the restrictions~$\Phi\mid_{\mathbb{C}_+}$ and~$\Phi\mid_{\mathbb{C}_-}$ are injective.
If~$\zeta_1 \in \mathbb{C}_+$ and~$\zeta_2 \in \mathbb{C}_-$, then~$\Phi(\zeta_1) \ne \Phi(\zeta_2)$
by~\eqref{a6} and the fact that~$\Phi$ is odd.

\paragraph{Case~$\alpha \ne \frac{k-1}{2}$,~$\sigma_1 = \tau_1$.} As in the case of even parameters, the integral in question equals~$\Phi'(\sigma_1)$. This value is non-zero, because it is a derivative of an injective analytic mapping.

\section{Non-elliptic case}\label{SVanish}
\subsection{Proof of Theorem~\ref{TVanishing}}
It is more convenient to work with compactly supported~$f$ and~$g$. The following claim is a straightforward consequence of the fact that~$\mathfrak{D}$ is dense in~$\mathcal{S}$.
\begin{Fact}\label{Compact support}
Suppose that inequality~\textup{\eqref{BilinearEmbeddings}} holds for some~$k,l,\alpha,\beta,\sigma,\tau$ with any~$f$ and~$g$ that belong to~$\mathfrak{D}(\mathbb{R}^2)$. Then~$\BE(k,l,\alpha,\beta,\sigma,\tau)$ also holds.
\end{Fact}
So, when proving Theorem~\ref{TVanishing}, we may assume that~$f$ and~$g$ are compactly supported. Moreover, using rescaling, we may make their support lie inside a unit disc centered at zero.

By symmetry, we may assume that~$k$ is odd. We remind the reader that~$\BE(k,l,\frac{k-1}{2},\frac{l-1}{2},\sigma,\tau)$ follows from the uniform boundedness of the kernel given by formula~\eqref{integrand}. However, in the case of real~$\sigma_1$ and~$\tau_1$, the situation is more complicated and the set~$\Omega_{\eps,R}$ is defined as
\begin{equation*}
\Omega_{\eps,R} = \Big\{(\xi,\eta)\in\mathbb{R}^2\big|\,\,\eps < |\eta| < R,\,\,|\xi^k - \sigma_1\eta^l| > \eps g(|\eta|),\,\,|\xi^k - \tau_1\eta^l| > \eps g(|\eta|)\Big\},
\end{equation*}
where~$g$ is some positive function decaying rapidly at zero and infinity. We note that by our assumption about the supports of~$f$ and~$g$ we may assume that~$|a| < 2$ and~$|b| < 2$, where~$a = x_1 - x_2$ and~$b = y_1 - y_2$. We may assume that~$a > 0$ and~$\eps < 1$.

We make our traditional change of variable~$\xi = \rho |\eta|^{\frac lk}$:
\begin{equation*}
\int\limits_{\Omega_{\eps,R}}\frac{\xi^{k-1}
e^{2\pi i (a\xi + b\eta)}\, d\xi}{(\xi^k - \tau_1 \eta^l)(\xi^k -
\sigma_1 \eta^l)}  = \frac{1}{(2\pi)^{2k}}\int\limits_{\varepsilon\leqslant |\eta|\leqslant R}
|\eta|^{-1}\int\limits_{B_{\eps}(\eta)}\frac{\rho^{k-1}e^{2\pi i(a\rho|\eta|^{l/k}+b\eta)}}{\big(\rho^k -
\tau_1(\sign\eta)^l\big)\big(\rho^k -\sigma_1(\sign\eta)^l\big)}\,d\rho d\eta.
\end{equation*}  
Here~$B_{\eps}(\eta)$ is given by
\begin{equation*}
B_{\eps}(\eta) = \big\{\rho \in \mathbb{R}\,\big|\,\,\dist(\rho^k,\{\sigma_1(\sign \eta)^l,\tau_1(\sign \eta)^l\}) >  \eps |\eta|^{-l}g(|\eta|)\big\}.
\end{equation*}
We take~$(2\pi a)^{\frac{k}{l}}\eta$ to be a new variable, redefine~$R$ and~$b$ (but we still carry~$\eps$ and~$a$ and do not change them; it will be convenient to use the restrictions~$\eps$ < 1 and~$a < 2$) and rewrite the integral as
\begin{equation}\label{2Int}
\int\limits_{(2\pi a)^{-\frac{k}{l}}\varepsilon\leqslant |\eta|\leqslant R}
|\eta|^{-1}\int\limits_{B_{\eps,a}(\eta)}\frac{\rho^{k-1}e^{ i(\rho|\eta|^{l/k}+b\eta)}}{\big(\rho^k -
\tau_1(\sign\eta)^l\big)\big(\rho^k -\sigma_1(\sign\eta)^l\big)}\,d\rho d\eta,
\end{equation}
where
\begin{equation*}
B_{\eps,a}(\eta) = \big\{\rho \in \mathbb{R}\,\big|\,\,\dist(\rho^k,\{\sigma_1(\sign \eta)^l,\tau_1(\sign \eta)^l\}) >  \eps |(2\pi a)^{-\frac{k}{l}}\eta|^{-l}g(|(2\pi a)^{-\frac{k}{l}}\eta|)\big\}.
\end{equation*}
The integral with respect to~$\rho$ can be represented as a contour integral (after rewriting~$\int_{\mathbb{R}} = \lim_{r \to \infty}\int_{-r}^r$), see Figure~\ref{fig:con}.
\begin{figure}[h]
\begin{center}
\includegraphics[width = 1 \linewidth]{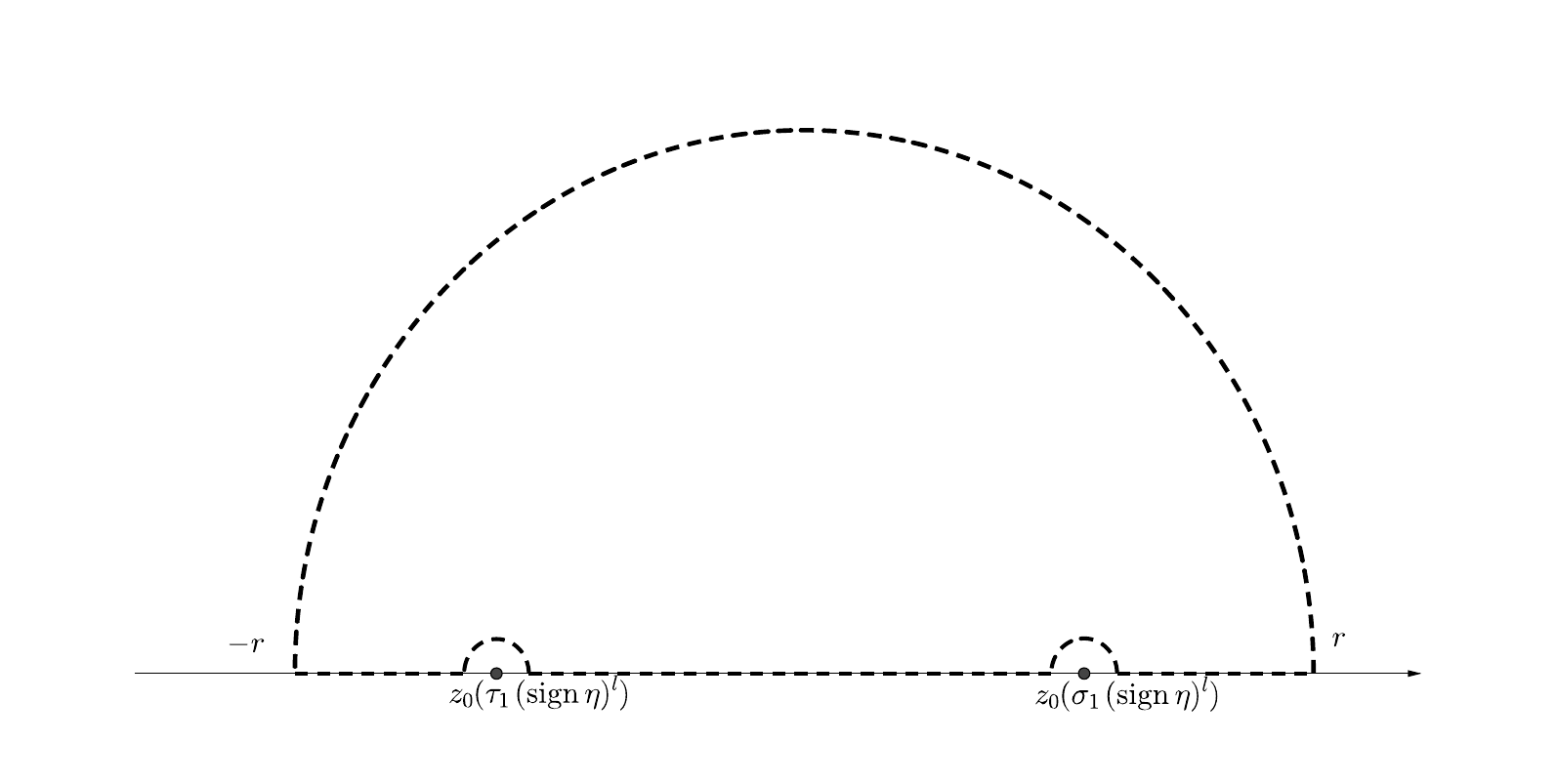}
\caption{Contour of integration.}
\label{fig:con}
\end{center}
\end{figure} 
We denote a unique real root of the equation~$z^k = s$,~$s \in \mathbb{R}$, by~$z_0(s)$. It is not hard to see that this equation has exactly~$\frac{k-1}{2}$ roots in the upper half-plane~$\Im z > 0$. So, if~$r$ is sufficiently large, there are exactly~$k-1$ poles of the integrand inside the contour ($\frac{k-1}{2}$ matching~$s = \sigma_1(\sign \eta)^l$ and~$\frac{k-1}{2}$ matching~$s = \tau_1 (\sign \eta)^l$). We also take the function~$g(\eta)$ to be so small that the small semicircles do not intersect ($g(\eta) \lesssim \eta^l$ will do because~$\eps < 1$). We need a simple lemma on estimating the difference between the integral over a small semicircle and the ``semi-residue''; this is a quantitative version of the Sokhotski--Plemelj formula. 
\begin{Le}\label{error}
If a meromorphic function~$\psi$ has a simple pole at a point~$x_0$ on the real line\textup, then
\begin{equation*}
\Bigg|\int\limits_{\substack{|z-x_0|=r \\ \Im z >0}}\psi(z)\,dz +\pi i\Res_{x_0}\psi\Bigg|\,
\leqslant\pi r\max_{\substack{|z - x_0|\leqslant r,\\\Im z >0}}|h'|,
\end{equation*}  
where~$h(z)=(z-x_0)\psi(z)$ \textup{(}the semicircle of integration is oriented clockwise\textup{)}.
\end{Le}
\begin{proof}
Since~$\psi$ has simple pole at~$x_0$, it can be written in the form~$\psi(z)=\frac{c}{z - x_0}+\psi_1(z)$, where $\psi_1$ is regular near $x_0$. Accordingly, we write
\begin{equation*}
\int\limits_{\substack{|z-x_0|=r \\ \Im z >0}}\!\!\!\!\!\!\!\psi(z)\,dz=\int\limits_{\substack{|z-x_0|=r \\ \Im z >0}}\!\!\!\!\!\frac{c\,dz}{z-x_0}+
\int\limits_{\substack{|z-x_0|=r \\ \Im z >0}}\!\!\!\!\!\psi_1(z)\,dz.
\end{equation*}
Since $\psi_1(z)=\frac{h(z)-h(x_0)}{z-x_0}$, the second integral is estimated as follows:
\begin{equation*}
\Bigg|\int\limits_{\substack{|z-x_0|=r \\ \Im z>0}}\psi_1(z)\,dz\Bigg|
\leqslant\pi r\max_{\substack{|z - x_0|\leqslant r,\\ \Im z >0}}|\psi_1|
\leqslant\pi r\max_{\substack{|z - x_0|\leqslant r,\\ \Im z >0}}|h'|.
\end{equation*}
The first integral can be calculated with ease:
\begin{equation*}
\int\limits_{\substack{|z - x_0| = r \\ \Im z > 0}} \frac{c\,dz}{z - x_0} =
\int\limits_{\substack{|z| = r,\\ \Im z > 0}} \frac{c\,dz}{z} =
\int\limits_{\frac{1}{2}}^0 \frac{c\,d(re^{2\pi i \theta})}{re^{2\pi i \theta}} =
-\pi i c.
\end{equation*}
\end{proof}
Therefore, the sum of the integrals over the semicircles is
\begin{equation*}
-\pi i e^{ib\eta}\frac{e^{ i z_0(\sigma_1 (\sign\eta)^{l})|\eta|^{\frac lk}} - e^{ i z_0(\tau_1 (\sign\eta)^{l})|\eta|^{\frac lk}}}{k(\sigma_1 - \tau_1)} + \eps O\bigg( \max(1,|\eta|^{\frac{l}{k}})|(2\pi a)^{-\frac{k}{l}}\eta|^{-l}g(|(2\pi a)^{-\frac{k}{l}}\eta|)\bigg).
\end{equation*}
The sum of residues result in a sum of integrals of the type~\eqref{SimpleIntegral}, as we have seen, such integrals are uniformly bounded. However, the integral that comes from semi-residues is subtler. We begin with estimating the error (i.e. the part of the integral~\eqref{2Int} that comes from the~$O$):
\begin{equation*}
\begin{aligned}
\eps\int\limits_{-\infty}^{\infty}\max(1,|\eta|^{\frac{l}{k}})|(2\pi a)^{-\frac{k}{l}}\eta|^{-l}g\big(|(2\pi a)^{-\frac{k}{l}}\eta|\big) \,\frac{d\eta}{|\eta|} = 
\eps\int\limits_{-\infty}^{\infty}\max\big(1,2\pi a|\eta|^{\frac{l}{k}}\big)|\eta|^{-l}g(|\eta|) \,\frac{d\eta}{|\eta|} \stackrel{\scriptscriptstyle a < 2}{\leq}\\
\eps \int\limits_{-\infty}^{\infty}\big(1 + |\eta|^{\frac{l}{k}}\big)|\eta|^{-l}g(|\eta|) \,\frac{d\eta}{|\eta|} \leq \eps
\end{aligned}
\end{equation*}
provided~$g(z) = |z|^{2l + 1} e^{-|z|}$ (we also multiply this function by a small constant to fulfill~$g(|\eta|)\lesssim\eta^l$). So, the contribution of the error to~\eqref{2Int} is uniformly bounded (and even small if~$\eps$ is small). The integral coming from the semi-residues looks like this (we have redenoted~$\eps$):
\begin{equation*}
\int\limits_{\eps}^{R} e^{i b \eta}\frac{\big(e^{i c_1|\eta|^{\frac lk}} - e^{i c_2|\eta|^{\frac lk}}\big)\,d\eta}{\eta},
\end{equation*}
here~$c_1$ and~$c_2$ are some real constants (equal to~$z_0(\sigma_1 (\sign\eta)^{l})$ and~$z_0(\tau_1 (\sign\eta)^{l})$ correspondingly).
The part of the integral over the interval~$[0,1]$ is bounded, so we can drop it. On the ray~$(1,\infty)$, we use triangle inequality
\begin{equation*}
\bigg|\int\limits_{1}^{R} e^{i b \eta}\frac{\big(e^{i c_1|\eta|^{\frac lk}} - e^{i c_2|\eta|^{\frac lk}}\big)\,d\eta}{\eta}\bigg| \leq \bigg|\int\limits_1^R e^{i (b\eta +c_1|\eta|^{\frac lk})}\,\frac{d\eta}{\eta}\bigg| + \bigg|\int\limits_1^R e^{i (b\eta + c_2|\eta|^{\frac lk})}\,\frac{d\eta}{\eta}\bigg|
\end{equation*}
and make a change of variable~$|c_{1,2}|^{\frac kl}\eta = e^x$, redefine~$b$, and arrive at
\begin{equation*}
\Big|\int\limits_0^{e^R} e^{i (b e^x \pm e^{\alpha x})} \,dx\Big|,
\end{equation*}
where~$\alpha = \frac{l}{k} \ne 1$.
It remains to prove a simple lemma on oscillatory integrals.
\begin{Le}\label{OscillatoryIntegral}
Let~$\alpha$\textup,~$\alpha \ne 1$\textup, be a positive fixed parameter. Then
\begin{equation*}
\Big|\int\limits_0^{R} e^{i (b e^x \pm e^{\alpha x})}\,dx\Big| \lesssim 1.
\end{equation*}
\end{Le}
This lemma will be deduced from the second Van-der-Corput lemma (for example, see~\cite{HE},~$\S 2.5.2$) cited below.
\begin{Le}\label{VdC}
Suppose that the function~$F$ is twice differentiable on~$(a,b)$ and~$F''$ does not have roots on this interval. In this case\textup,
\begin{equation*}
\Big|\int\limits_a^b e^{iF(x)}\,dx \Big| \leq \frac{8 \sqrt{\pi}}{\sqrt{\min\limits_{\xi \in (a,b)} |F''(\xi)|}}.
\end{equation*}
\end{Le}

\paragraph{Proof of Lemma~\ref{OscillatoryIntegral}.} 
Consider the case where~$\pm$  is~$-$, the remaining case is similar.
For brevity, we introduce the functions~$h_b$ given by~$h_b(x) = be^x - e^{\alpha x}$. Surely,~$h''_b(x) = be^x - \alpha^2e^{\alpha x}$, which can be zero (and thus does not allow to apply Lemma~\ref{VdC} directly). However,~$h''_b(x)$ can be small only on a set of small measure. It is enough to prove that
\begin{equation*}
\Big|\int\limits_C^{R} e^{i (b e^x - e^{\alpha x})}\,dx\Big| \lesssim 1,
\end{equation*}
where~$C$ is some numeric constant that does not depend on~$b$ (but will depend on~$\alpha$). The function~$h''_b$ changes monotonicity on~$[C,\infty)$ at most once. For a monotone function, the set where its modulus does not exceed one, is an interval (or ray). Therefore, the set~$\{x\in[C,\infty) \mid |h''_b(x)| < 1\}$ is a union of at most two intervals (this set is finite, because~$|h''_b|(x) \to \infty$ as~$x\to \infty$). We are going to prove that if~$|h''_b(z)| < 1$, then~$|h''_b(z\pm 1)| > 1$ for~$z \in [C,\infty)$. If this assertion is proved, it is not hard to see that the intervals constituting the set~$\{x\in[C,\infty)\mid |h''_b(x)| < 1\}$ have common length at most~$2$. The complement of this set in~$[C,\infty)$ is a union of at most three intervals (one of them is a ray). On each of them, the integral can be estimated by~$20$ by Lemma~\ref{VdC}, on the complement it is estimated by~$2$ and the lemma is proved. 

To prove the assertion, we denote~$b e^z$ by $p$, and~$\alpha^2e^{\alpha z}$ by $q$. If both~$|h''_b(z)| < 1$ and~$|h''_b(z + 1)| < 1$ (we consider this case, the case~$|h''_b(z-1)| < 1$ is similar), then
\begin{equation*}
q-1 < p < q+1\quad \hbox{and}\quad e^{\alpha} q -1< ep < e^{\alpha} q + 1.
\end{equation*}
In this case,~$e^{\alpha} q - 1 < e(q+1)$ and~$e(q-1) < e^{\alpha} q + 1$, which is~$q< \frac{e + 1}{|e^{\alpha} - e|}$. Taking~$C >10 +  \alpha^{-1}\ln\frac{e + 1}{\alpha^2|e^{\alpha} - e|}$, we get a contradiction. The assertion is proved.
\qed

\subsection{Proof of Corollary~\ref{EvenParametersVanishingOp}}\label{SsEvenParam}
Corollary~\ref{EvenParametersVanishingOp} will be obtained by a simple algebraic trick.
\begin{Le}\label{Implication}
The following implication holds provided~$\sigma\tau \ne 0$\textup:
\begin{equation*}
\begin{aligned}
\BE(k,l,\alpha,\beta,\sigma,\tau)\,\, \&\,\,\BE(k,l,\alpha,\beta,-\sigma,-\tau) \,\,\&\,\,\BE(k,l,\alpha,\beta,\sigma,-\tau)\,\, \&\,\, \BE(k,l,\alpha,\beta,-\sigma,\tau) \Rightarrow\\ \BE(2k,2l,\alpha + 2k,\beta,\sigma^2,\tau^2)\,\,\&\,\,\BE(2k,2l,\alpha,\beta + 2l,\sigma^2,\tau^2).
\end{aligned}
\end{equation*}
\end{Le}
\begin{proof}
By applying the statement~$\BE(k,l,\alpha,\beta,\pm\sigma,\pm\tau)$ to the pair of functions~$(\partial_1^k \pm\sigma\partial_2^l)f$ and~$(\partial_1^k \pm \tau\partial_2^l)g$, we get estimates on four scalar products:
\begin{equation*}
\Big|\scalprod{(\partial_1^k \pm \tau\partial_2^l)f}{(\partial_1^k \pm \sigma \partial_2^l )g}_{W^{\alpha,\beta}_2}\Big| \lesssim \big\|(\partial_1^{2k} - \tau^2\partial_2^{2l})f\big\|_{L_1}\big\|(\partial_1^{2k} - \sigma^2\partial_2^{2l})g\big\|_{L_1}.
\end{equation*}
It is not hard to see that one can express both~$\scalprod{\partial_1^k f}{\partial_1^k g}_{W^{\alpha,\beta}_2}$ and~$\scalprod{\partial_2^l f}{\partial_2^l g}_{W^{\alpha,\beta}_2}$
as a linear combination of the four scalar products on the left-hand side, and thus estimate them by the expression on the right-hand side. We only have to use equation~\eqref{DerivationOfScalarProduct} to obtain~$\BE(2k,2l,\alpha + 2k,\beta,\sigma^2,\tau^2)$ and~$\BE(2k,2l,\alpha,\beta + 2l,\sigma^2,\tau^2)$.
\end{proof}
\paragraph{Proof of Corollary~\ref{EvenParametersVanishingOp}.} By Theorem~\ref{TVanishing}, the statments~$\BE(\frac{k}{2},\frac{l}{2},\frac{k}{4} - \frac12,\frac{l}{4}-\frac12,\pm\sigma',\pm\tau')$, where~$\sigma'_1 = \sqrt{\sigma_1}$ and~$\tau_1'=\sqrt{\tau_1}$, hold true (because one of the numbers~$\frac{k}{2}$ and~$\frac{l}{2}$ is odd and the numbers~$\pm\sigma_1'$ and~$\pm\tau_1'$ are real distinct). Lemma~\ref{Implication} finishes the proof.
\subsection{Proof of Lemma~\ref{KnappExample}}
We are going to disprove even a weaker statement, namely, we disprove the inequality
\begin{equation}\label{LinearCoercive}
\|f\|_{W_2^{\alpha,\beta}} \lesssim \|(\partial_1^k - \sigma\partial_2^l)f\|_{L_1}.
\end{equation}
The case where~$(k,l,\sigma,\sigma)$ is elliptic, has been already considered. We assume that~$(k,l,\sigma,\sigma)$ is not elliptic.

Let~$\zeta = (\zeta_1,\zeta_2)$ be some non-zero point in~$\mathbb{R}^2$ such that~$\zeta_1^k - \sigma_1\zeta_2^l = 0$ and~$|\zeta_1| > 2$,~$|\zeta_2| > 2$. We will consider only the functions~$f$ whose spectrum lies in a~$\frac12$-neighbourhood of~$\zeta$; for such functions~$f$
\begin{equation*}
\|f\|_{W_2^{\alpha,\beta}} \asymp \|f\|_{L_2}
\end{equation*}
because the weight~$|\xi|^{2\alpha}|\eta|^{2\beta}$ in formula~\eqref{norma} is bounded from below and above on the support of~$\hat{f}$. 
Let~$\phi \in \mathfrak{D}(\mathbb{R})$ be a real-valued function supported on~$[-\frac12,\frac12]$. By~$\phi_{\lambda}$ we denote the function~$x \to \phi(\lambda x)$. The sequence~$\{f_n\}_n$ that disproves~\eqref{LinearCoercive} is given by formula
\begin{equation*}
\hat{f}_n(\xi,\eta) = \phi_n(\eta - \zeta_2)\phi_{n^2}(\xi - (\sigma_1\eta^l)^{\frac{1}{k}}), \quad |\xi - \zeta_1| \leq 1, |\eta - \zeta_2| \leq 1,
\end{equation*}
and zero when~$\xi$ or~$\eta$ in all other cases. We see that the function~$\hat{f}_n$ is supported on a small~$n^{-1}\times n^{-2}$-rectangle~$R_n$ that covers the~$n^{-2}$-neighbourhood of the curve~$\gamma = \{(\xi,\eta)\in\mathbb{R}^2\mid \xi^k = \sigma_1\eta^l\}$ near the point~$\zeta$ (namely, the bigger side of~$R_n$ is parallel to the tangent to~$\gamma$ at the point~$\zeta$). The~$L_2$-norm on the left-hand side of~\eqref{LinearCoercive} can be computed with ease,
\begin{equation}\label{L2estimate}
\|f_n\|_{L_2}^2 = \|\hat{f}_n\|_{L_2}^2 = \!\!\!\!\int\limits_{|\eta - \zeta_2| < \frac1n}\!\!\!\!\!\!\!\!\phi_n^2(\eta - \zeta_2)\!\!\!\!\!\!\!\!\int\limits_{|\xi - \zeta_1| < \frac12}\!\!\!\!\!\!\!\!\phi_{n^2}^2(\xi - (\sigma_1\eta^l)^{\frac{1}{k}})\,d\xi d\eta= \frac{1}{n^2}\|\phi\|_{L_2}^2\int\limits_{\mathbb{R}}\phi_n^2(\eta - \zeta_2) \,d\eta = \frac{1}{n^3}\|\phi\|_{L_2}^4.
\end{equation}
If~$n$ is big enough, then
\begin{equation*}
\big|(\xi^k - \sigma_1\eta^l)\hat{f}_n(\xi,\eta)\big| \lesssim n^{-2},
\end{equation*}
because on~$R_n$, where~$\hat{f}_n$ is supported, the polynomial~$(\xi^k - \sigma_1\eta^l)$ does not exceed~$n^{-2}$. Consequently,
\begin{equation*}
\big\|(\partial_1^k - \sigma\partial_2^l)f_n\big\|_{L_{\infty}} \lesssim \big\|(\xi^k - \sigma_1\eta^l)\hat{f}_n\big\|_{L_1} \lesssim n^{-5},
\end{equation*}
because the latter function does not exceed~$n^{-2}$ and is supported on~$R_n$ whose measure is~$n^{-3}$. Consider the polar set~$R^{\circ}_n$ of~$R_n$. This is a rectangle~$n\times n^2$ whose ``longer axis'' is perpendicular to the ``longer axis'' of~$R_n$ (it is parallel to the normal at~$\zeta$ to the curve~$\{\xi^k = \sigma_1\eta^l\}$). 
\begin{Fact}
For any~$\eps > 0$ and any~$r > 0$ fixed we have an estimate
\begin{equation*}
\big|(\partial_1^k - \sigma\partial_2^l)f_n(x,y)\big| \lesssim n^{-3}(x^2 + y^2)^{-r},\quad~\dist\big((x,y),R_n^{\circ}\big) \geq n^{\eps}.
\end{equation*}
\end{Fact} 
The proof of this fact is rather standard, but takes some place if written in detail. It is the same as the proof of the fact ``the Fourier transform of an~$r$ times continuously differentiable compactly supported function decays at infinity as~$|\xi|^{-r}$'' (one has to do lots of integrations by parts). With this fact at hand ($B_r(z)$ stands for the ball of radius~$r$ centred at~$z$), 
\begin{equation*}
\begin{aligned}
\iint\limits_{\mathbb{R}^2}\big|(\partial_1^k - \sigma\partial_2^l)f_n\big| = \int\limits_{R_n^{\circ} + B_{n^{\eps}}(0)}\!\!\!\!\!\big|(\partial_1^k - \sigma\partial_2^l)f_n\big| +\!\!\!\!\! \int\limits_{\mathbb{R}^2 \setminus (R_n^{\circ} + B_{n^{\eps}})}\!\!\!\!\!\big|(\partial_1^k - \sigma\partial_2^l)f_n\big| \lesssim\\
n^{-5}n^{3 + 2\eps} + \!\!\!\!\!\int\limits_{\mathbb{R}^2 \setminus (R_n^{\circ} + B_{n^{\eps}})}\!\!\!\!\! n^{-3}(x^2 + y^2)^{-r}\,dx\,dy \lesssim n^{-2 + 2\eps}.
\end{aligned}
\end{equation*}
This inequality, together with~\eqref{L2estimate}, contradicts~\eqref{LinearCoercive} when~$n$ is large enough (we have~$n^{-\frac32}$ on the left-hand side and~$n^{-2 + 2\eps}$ on the right).

\section{Further development}\label{SFurther}
Though we were studying bilinear inequalities, we got some additional information about quadratic inequalities of the form
\begin{equation}\label{QuadraticEmbeddings}
\|f\|_{W_2^{\alpha,\beta}(\mathbb{R}^2)}^2\lesssim\|(\partial_1^k - \sigma\partial_2^l)f\|_{L_1(\mathbb{R}^2)} \|(\partial_1^k - \tau\partial_2^l)f\|_{L_1(\mathbb{R}^2)}.
\end{equation}
The statement that inequality~\eqref{QuadraticEmbeddings} holds true will be called~$\QE(k,l,\alpha,\beta,\sigma,\tau)$. Surely,~$\BE$ implies~$\QE$ with the same parameters. So, our positive results for~$\BE$ transfer to~$\QE$. Moreover, by an application of inequality between arithmetic and geometric means to formula~\eqref{norma} we see that
\begin{equation*}
\|f\|_{W_2^{\alpha,\beta}} \lesssim \|f\|_{W_2^{\alpha^-,\beta^-}} + \|f\|_{W_2^{\alpha^+,\beta^+}},\quad (\alpha_-,\beta_-)\,\,\hbox{and}\,\,(\alpha_+,\beta_+)\,\,\hbox{satisfy~\eqref{line}}\,\, \hbox{and}\,\,\alpha \in (\alpha_-,\alpha_+).
\end{equation*} 
This, together with Corollary~\ref{EvenParametersVanishingOp} shows that~$\QE$ holds true when~$k$ and~$l$ are even, but one of~$\frac{k}{2}$ and~$\frac{l}{1}$ is odd,~$\sigma_1$ and~$\tau_1$ are real positive,~$(\alpha,\beta)$ satisfies~\eqref{line} and lies between~$(\frac34 k - \frac12,\frac14 l -\frac12)$ and~$(\frac14 k -\frac12,\frac34 l -\frac12)$. 

The negative results for~$\QE$ are poorer. We surely know that~$\QE$ does not hold when~$\sigma = \tau$ (this follows from the corresponding part of the proof of Theorem~\ref{T3} and Lemma~\ref{KnappExample}). It seems that the quadratic inequalities resemble standard embedding theorems more than the bilinear ones.
\begin{Conj}
The statement~$\QE$ holds whenever~$\sigma \ne \tau$.
\end{Conj}

Our positive results for~$\BE$ are similar in elliptic and non-elliptic cases. The situation for negative results is far from being the same. The main difficulty in constructing counter-examples in the style of Subsection~\ref{SsCounterexamples} for the non-elliptic case is explained by the following fact.
\begin{Fact}\label{Spaces}
Consider the closure in~$L_p(\mathbb{R}^2)$\textup,~$1 \leq p < \frac43$\textup, of the set
\begin{equation*}
\Big\{(\partial_1^k - \sigma\partial_2^l)f \,\Big|\,\, \hat{f}\in\mathfrak{D}(\mathbb{R}^2)\Big\}
\end{equation*}
and denote this linear space by~$G_p^{k,l,\sigma}$. Suppose~$k\ne l$. If the polynomial~$\xi^ k -\sigma_1\eta^l$ is elliptic\textup, this space coincides with the space~$L_p(\mathbb{R}^2)$ provided~$p > 1$ and with the space~$L_1^0(\mathbb{R}^2)$ that consists of function having zero integral provided~$p=1$. When~$\xi^ k -\sigma_1\eta^l$ is not elliptic\textup, the space~$G_p^{k,l,\sigma}$ is not complemented in~$L_p$.
\end{Fact}
We will not prove this fact. The first part is standard. We give an explanation about the nature of the effect hidden in the second part. Consider the set~$\gamma = \{(\xi,\eta)\in\mathbb{R}^2\mid \xi^k = \sigma_1\eta^l\} \subset \mathbb{R}^2$. Out of zero, this is a smooth~$C^{\infty}$ curve (a submanifold of~$\mathbb{R}^2$), moreover, this curve is convex. If~$p < \frac43$ and~$K$ is some compact subset of~$\gamma$ (we suppose that~$0 \notin K$), then the operator defined by formula
\begin{equation*}
R_{K}[f] = \hat{f}\big|_{K},\quad f \in \mathcal{S}(\mathbb{R}^2),
\end{equation*}
is continuous from~$L_p(\mathbb{R}^2)$ to~$L_1(K)$ (we equip~$K$ with the Lebesgue measure on~$\gamma$), see e.g.~\cite{T}. Therefore, the kernel of this operator is closed in~$L_p$, denote it by~$L_p^K$. The intersection of all such spaces is again closed in~$L_p$, call it~$L_p^{\gamma}$. Surely,~$G_p^{k,l,\sigma} \subset L_p^{\gamma}$. It is a small surprise, but~$G_p^{k,l,\sigma} = L_p^{\gamma}$. This is a consequence of the fact that if~$\varphi \in \mathfrak{D}(\mathbb{R}^2)$ and~$\varphi 
\equiv  0$ on~$\gamma$, then~$(\xi^k - \sigma_1\eta^l)^{-1}\varphi$ is also in~$\mathfrak{D}(\mathbb{R}^2)$, at least when~$0 \notin \supp \varphi$ (some problems with zero appear, but they are negligible for the spaces of functions). The fact that~$L_p^{\gamma}$ is not complemented follows from a well-known principle (see, e.g.~\cite{R}, Lemma~$3.1$ for~$p > 1$ and Theorem~$1.4$ for the case~$p=1$) that if a translation-invariant subspace of~$L_p$ is complemented, then there is a projector onto it that is also translation-invariant (i.e. a multiplier). 

We state, again without proof, that when~$\frac43 \leq p < \infty$, the spaces~$L_p$ and~$L^{\gamma}_p = G_p^{k,l,\sigma}$ coincide. 

Though Fact~\ref{Spaces} forbids counter-examples in the style of Subsection~\ref{SsCounterexamples}, it may give some additional hope for embedding theorems to hold. We end the story with a small discussion of a related conjecture.
\begin{Conj}
Let the set~$(k,l,\sigma,\sigma)$ be non-elliptic. The inequality
\begin{equation*}
\|f\|_{W_q^{\alpha,\beta}} \lesssim \|(\partial_1^k - \sigma \partial_2^l)f\|_{L_p}
\end{equation*}
holds true whenever~$\frac{1}{p} - \frac{1}{q} \geq \frac23$\textup,~$1 < p < \frac43$\textup,~$q < \infty$ and
\begin{equation}\label{HomogeneityCondition}
\frac{\alpha}{k} + \frac{\beta}{l} = 1 - \Big(\frac{1}{p} - \frac{1}{q}\Big)\Big(\frac{1}{k} + \frac{1}{l}\Big).
\end{equation}
\end{Conj}
The space on the left-hand side of the inequality is defined as the closure of~$\mathcal{S}(\mathbb{R}^2)$ with respect to the norm
\begin{equation*}
\|f\|_{W_q^{\alpha,\beta}} = \Big\|\mathcal{F}^{-1}\big[\mathcal{F}[f](\xi,\eta)|\xi|^{\alpha}|\eta|^{\beta}\big] \Big\|_{L_q},
\end{equation*}
where~$\mathcal{F}$ stands for the Fourier transform. When the set~$(k,l,\sigma,\sigma)$ is elliptic, the statement holds true whenever~$p > 1$ and follows from classical embedding theorems and the fact that pure derivatives are expressed in terms of~$(\partial_1^k - \sigma \partial_2^l)$ by a singular integral operator.

We are able to prove the conjecture with additional assumption~$q > 4$. This is a combination of known facts. First, the problem can be localized in spectrum, i.e. reduced to the case where~$\hat{f}$ is supported near some point on the curve~$\gamma$. This is done by decomposing~$f = \sum P_jf$, where~$P_j$ is the Littlewood--Paley smoothed spectral projector, application of Littlewood--Paley inequality~$\|f\|_{L_p} \asymp \|(\sum |P_j f|^2)^{\frac12}\|_{L_p}$, H\"older inequality (for the last step, the inequalities~$p \leq 2$ and~$q \geq 2$ are of crucial importance), and anisotropic dilatation invariance of the problem given by equation~\eqref{HomogeneityCondition}. After the problem has been localized, the space on the left-hand side turns into the standard Lebesgue space~$L_q$. Thus, we are studying the continuity of the operator with the symbol~$\frac{1}{\xi^k - \sigma_1\eta^l}$ as an operator from the localized space~$L_p^{\gamma}$ to~$L_q$. This is, modulo convolutions with Schwartz functions, the same as to study the operator with the symbol~$\vp(\dist(\cdot,\gamma))^{-1}$. The continuity of a similar operator, but between standard Lebesgue spaces, had been completely studied in~\cite{B}. Namely, in this paper, the negative powers of the Bochner-Riesz operator were studied (i.e. the same-type operators, but where the degree~$-1$ can vary from~$-\frac32$ to~$0$ and even be complex, but~$\gamma$ is a circumference). However, the method can be applied verbatim to the case of arbitrary smooth convex curve in place of the unit circumference. It finishes the proof of the conjecture for the cases where~$q > 4$.

However, it is not hard to see that when~$q < 4$, the corresponding operator is not continuous as an operator between the Lebesgue spaces. So, the interesting cases of the conjecture are where~$q  <4$. We believe that this is the place where the space~$L^{\gamma}_p$ can play a definite role.

Dmitriy M. Stolyarov

St. Petersburg Department of Steklov Mathematical Institute, Russian Academy of Sciences (PDMI RAS);
P. L. Chebyshev Research Laboratory, St. Petersburg State University.

dms@pdmi.ras.ru
	

\begin{thebibliography}{99}

\bibitem{B} J.-G.~Bak, \emph{Sharp estimates for the Bochner--Riesz operator of negative order in~$\mathbb{R}^2$}, Proceedings of the American Mathematical Society {\bf 125} (1997), n.7, 1977--1986.

\bibitem{BIN} O.V.~Besov, V.P.~Il'in, S.M.~Nikolski, \emph{Integral representations of functions and embedding theorems}, 1975.

\bibitem{HE} V. Havin, B. J\"oricke, \emph{The Uncertainty principle in harmonic analysis}, Springer, 1994.

\bibitem{H} L. H\"ormander, \emph{On the division of distributions by polynomials}, Ark. fur Mat., {\bf 3} 6 (1958), 555--568.

\bibitem{KMS} S. V. Kislyakov, D. V. Maksimov, D. M. Stolyarov, \emph{Spaces of Smooth Functions Generated by Nonhomogeneous Differential Expressions}, Funkts. Anal. Prilozh., {\bf 47}:2 (2013), 89--92.

\bibitem{KMSpreprint} S. V. Kislyakov, D. V. Maksimov, D. M. Stolyarov, \emph{Differential expressions with mixed homogeneity 
and spaces of smooth functions they generate}, http://arxiv.org/abs/1209.2078.

\bibitem{KS} S. V. Kislyakov, N. G. Sidorenko, \textit{Absence of a local unconditional
structure in anisotropic spaces of smooth functions}, Sibirsk. Mat. Zh., {\bf 29}, no. 3 (1988),
64-77 (Russian).

\bibitem{Kol} V. I. Kolyada, \emph{On an embedding of Sobolev spaces}, Mat. Zametki, {\bf 54}:3 (1993), 48--71 (Russian).

\bibitem{PS} A. Pelczynski, K. Senator, \emph{On isomorphisms of anisotropic Sobolev spaces with ``classical Banach spaces'' and a Sobolev type embedding theorem}, Studia Math., {\bf 84} (1986), 169--215.

\bibitem{R} H. P.~Rosenthal, \emph{Projections onto translation-invariant subspaces of~$L^p(G)$}, Memoirs of the American Mathematical Society {\bf 63}, 1966.

\bibitem{S} V. A. Solonnikov, \textit{On some inequalities for functions in the classes}
$\vec{W}_p(R^n)$, Zapiski Nauchn. Semin. LOMI, {\bf 27} (1972), 194--210.

\bibitem{T} T.~Tao, \emph{Recent progress on the Restriction conjecture}, http://arxiv.org/abs/math/0311181.

\end{thebibliography}
\end{document}